%% file: main_arxiv.tex
\documentclass[12pt]{amsart}

\usepackage{geometry}

\input{eng}

\input{eng_theorems}
\input{packages}
\input{preamble}

\input{preamble_extra}

\begin{document}

 \input{mains/norm-draft}

\printbibliography

\end{document}

%% file: eng.tex
\usepackage[english]{babel}
\usepackage[T1]{fontenc}
\usepackage[utf8]{inputenc}


%% file: eng_theorems.tex
\theoremstyle{plain}
    \newtheorem{theorem}{Theorem}[section]
    \newtheorem{lemma}[theorem]{Lemma}
    \newtheorem{proposition}[theorem]{Proposition}
    \newtheorem{corollary}[theorem]{Corollary}

\theoremstyle{definition}

    \newtheorem{example}[theorem]{Example}

\usepackage[
    draft = false,
    unicode = true,
    colorlinks = true,
    allcolors = blue,
    hyperfootnotes = true,
    citecolor = red
]{hyperref}

%% file: packages.tex
\usepackage{amsfonts,amsmath,amssymb,amscd,amsthm,url}
\usepackage[inline]{enumitem}
\usepackage{graphicx,epsf,afterpage}
\usepackage{soul}
\usepackage{multicol}
\usepackage{array}
\usepackage{braket}
\usepackage{epigraph}
\usepackage{rotating}

\usepackage[
backend=biber,
style=numeric-comp,
sorting=none,
giveninits=true
]{biblatex}
\usepackage{csquotes}

\usepackage{xcolor}

\usepackage{ulem}

\usepackage{tikz}
\usetikzlibrary{positioning, calc, intersections, through, arrows, matrix, chains, math}
\usetikzlibrary{shadows}
\usetikzlibrary{graphs}
\usetikzlibrary{graphs.standard}
\usetikzlibrary{patterns}
\usetikzlibrary{decorations.pathreplacing,calligraphy,backgrounds}

%% file: preamble.tex
\newcommand\norm[1]{\ensuremath{\left\lVert#1\right\rVert}}
\newcommand\abs[1]{\ensuremath{\left\lvert#1\right\rvert}}

\renewcommand{\Pr}{\mathrm{P}}

\newcommand{\Lcal}{\mathcal{L}}

\newcommand{\Xcal}{\mathcal{X}}

\newcommand{\R}{\ensuremath{\mathbb{R}}}
\newcommand{\Z}{\ensuremath{\mathbb{Z}}}
\newcommand{\Cx}{\ensuremath{\mathbb{C}}}

\newcommand{\N}{\ensuremath{\mathbb{N}}}

\renewcommand{\geq}{\geqslant}

\renewcommand{\leq}{\leqslant}

\newcounter{mcnt}

\newcounter{wordcnt}

%% file: preamble_extra.tex
\usepackage{comment}
\usepackage{bm}
\usepackage{microtype}
\newcommand{\ps}[1]{\left(#1\right)}

\renewcommand{\set}[1]{\left\{#1\right\}}

\newcommand{\mc}[1]{\mathcal{#1}}
\newcommand{\Ex}{\mathbb{E}}

\setenumerate{leftmargin=*}
\setitemize{leftmargin=*}

%% file: mains/norm-draft.tex
\title{Strong law of large numbers \\ for random semigroups in norm topology}

\author{A.\,E. Komagorov}

\begin{abstract}
    This paper studies the strong law of large numbers for products of random uniformly continuous semigroups in Banach spaces. We propose a novel approach by treating the product of semigroups as a solution to an integral equation and establish a version of the SLLN in norm topology for arbitrary Banach spaces.
\end{abstract}

\thanks{\hspace{-4mm}A.\,E. Komagorov: komagorov.ae@phystech.edu. orcid: 0009-0005-4588-9223? \\
Moscow Institute of Physics and Technology 141701, Institutskii per. 9, Dolgoprudny, Russia}

\maketitle
\thispagestyle{empty}


\vspace{3mm}

\section{Introduction}

The theory of compositions of random linear operators can nowadays be considered classical \cite{Mehta, Furstenberg, Tutubalin, Skorokhod}. 
Besides the purely theoretical interest, it has applications in random quantum dynamics, where either compositions of random linear operators in Hilbert spaces \cite{Kempe, OSS19} or compositions of random quantum channels \cite{Aharonov-D-Z, Kholevo-qprobstat, GOSS-2022, Pathirana-Schenker, Dhamapurkar-Dahlsten} are considered.

The strong law of large numbers (SLLN) for random variables acting in a Banach space is well studied \cite{Ledoux-Talagrand-1991-probability, Hoffman-Jorgensen-Pisier}.
Distributions of compositions of random linear operators were studied in \cite{Oseledets, Tutubalin, Skorokhod, Pathirana-Schenker}.
In particular, the law of large numbers for compositions of random operators in different operator topologies was studied in \cite{Sakbaev2016-lln, Sakbaev2018-avg, OSS19}.
Most of the results of this type follow the classical approach of the Chernoff product formula \cite{Chernoff-1968, Engel-Nagel-sec3, Galkin-Remizov-2025}, where the convergence of semigroups is typically described in strong operator topology (SOT) and occasionally in weak operator topology (WOT).
For example, in \cite{DzhenzherSakbaev25-SLLN}, the SLLN in SOT was obtained for uniformly continuous random semigroups in uniformly smooth spaces; and in \cite{DzhenzherSakbaev26-unbound}, a generalization was obtained for strongly continuous random semigroups.
However, these and earlier papers only consider uniformly bounded random operators and usually rely on some geometric properties of the underlying Banach space (such as uniform smoothness), which limits the applicability of such results.

In this paper, we present another viewpoint on the product of semigroups by considering it as the solution of an integral equation.
This allows us to derive the SLLN for random semigroups from the SLLN for their generators by employing the standard theorems of calculus for the Bochner integral.
In particular, this approach produces the SLLN in norm topology for all independent, identically distributed (i.i.d.) random uniformly continuous semigroups in an arbitrary Banach space. This directly generalizes earlier results obtained in \cite{DzhenzherSakbaev25-SLLN}.

The paper is organized as follows.
In Section~\ref{s:main}, the necessary notions and main results (Theorem~\ref{t:slln-norm} and Corollary~\ref{cor:slln-norm}) are formulated.
In Section~\ref{s:proof}, the proof of Theorem~\ref{t:slln-norm} is given.
Section~\ref{s:counterexamples} is devoted to counterexamples showing the tightness of the assumptions in Theorem~\ref{t:slln-norm}.
Finally, Section~\ref{s:conclusions} is devoted to some practical considerations and further ideas regarding our results.

\section{Background and main results}\label{s:main}

Below are the conventions adopted in this paper.
\begin{itemize}
    \item
    We consider Banach spaces over the field $\R$ or $\Cx$.
    \item
    We work with a probability space $(\Omega, \sigma_\Omega, \Pr)$ for a $\sigma$-additive probability measure $\Pr$.
    \item We say that some property $S \subseteq \Omega$ holds almost surely if there is $E \in \sigma_{\Omega}$ such that $E \subseteq S$ and $\Pr(E) = 1$ (this clarification addresses the case of a non-complete measure).
    \item
    $\N$ denotes the set of positive integers.
\end{itemize}

For a Banach space $\mathcal{X}$, we say that $x\colon\Omega\to \mathcal{X}$ is a \textbf{\emph{random element}} if it is Bochner measurable. For a Bochner integrable random element $x$, $\Ex x$ denotes the Bochner integral of $x$. A random element $A\colon\Omega\to \Lcal(\Xcal)$ is called a \textbf{\emph{random operator}} with the corresponding \textbf{\emph{random semigroup}} $e^{At}$. See \cite{DzhenzherSakbaev25-SLLN}, \cite[Chapter~2]{DiestelUhl} for some basic properties of random elements and random operators. 

There is a classical result about the strong law of large numbers in Banach spaces:

\begin{theorem}[Kolmogorov--Mourier SLLN {\cite{Mourier1953}, \cite[Corollary~7.10]{Ledoux-Talagrand-1991-probability}}]\label{t:kolm-slln}
    Let \(\Xcal\) be a separable Banach space.
    Let $\set{x_n\colon\Omega\to\Xcal}_{n\in\N}$ be a sequence of i.i.d. Bochner integrable random elements. Then almost surely
    \[
        \sum_{i=1}^n\frac{x_i}{n} \to \Ex x_1.
    \]
\end{theorem}

Actually, it is redundant to require separability of the Banach space in SLLN-type theorems, provided that the random elements are Bochner measurable (which some authors do not assume as part of the definition). This follows from the fact that a Bochner measurable random element is \emph{essentially separably valued}, i.e., takes values in a fixed separable subset of $\Xcal$ almost surely. Indeed, let $E_n \in \sigma_{\Omega}$ and $\mc{S}_n \subseteq \Xcal$ be an event and a separable subset such that $\Pr(E_n) = 1$ and $x_n(E_n) \subseteq \mc{S}_n$. Then, almost surely, all the $x_n$ take values in the closed linear span of $\bigcup_{n\in\N}\mc{S}_n$, which is still separable.

By contrast, it is important that the random elements are identically distributed to guarantee the SLLN in arbitrary Banach spaces. A generalization of the SLLN to independent random elements with bounded second moments requires B-convexity of the Banach space \cite{Beck1962}.

We define the SLLN for random semigroups in three (generally distinct) operator topologies. Let $\Xcal$ be a Banach space and let $\set{A_n\colon\Omega\to\Lcal(\mathcal{X})}_{n\in\N}$ be a sequence of Bochner integrable random operators.

\begin{itemize}
    \item The sequence of random semigroups $\{e^{A_nt}\}_{n \in \N}$ satisfies the \textbf{\emph{SLLN in norm topology}} if almost surely for every $T > 0$
    \[
        \lim_{n\to\infty} \sup\limits_{t \in [0,T]} \norm{e^{A_1t/n} \ldots e^{A_{n}t/n} - e^{\Ex A_1 t/n} \ldots e^{\Ex A_n t/n}} = 0.
    \]
    \item The sequence of random semigroups $\{e^{A_nt}\}_{n \in \N}$ satisfies the \textbf{\emph{SLLN in SOT}} if for every $x \in \Xcal$ almost surely for every $T > 0$
    \[
        \lim_{n\to\infty} \sup\limits_{t \in [0,T]} \norm{\left(e^{A_1t/n} \ldots e^{A_{n}t/n} - e^{\Ex A_1 t/n} \ldots e^{\Ex A_n t/n}\right)x} = 0.
    \]
    \item The sequence of random semigroups $\{e^{A_nt}\}_{n \in \N}$ satisfies the \textbf{\emph{SLLN in WOT}} if for every $x \in \Xcal$ and $f \in \Xcal^*$ almost surely for every $T > 0$
    \[
        \lim_{n\to\infty} \sup\limits_{t \in [0,T]} \abs{f\left(\ps{e^{A_1t/n} \ldots e^{A_{n}t/n} - e^{\Ex A_1 t/n} \ldots e^{\Ex A_n t/n}}x\right)} = 0.
    \]
\end{itemize}

The earlier papers usually reverse the order of quantifiers ``for all $T > 0$'' and ``almost surely''. These formulations are equivalent, which follows from continuity of the probability and the fact that supremum over the segment $[0, T]$ is monotone in $T$.

Also, it natural to inquire whether the SLLN for random semigroups survives reversing the order of composition of semigroups. In the setting of Theorem \ref{t:slln-norm}, the reversed SLLN can indeed be established using the same proof strategy. However, in general, the two convergences are not equivalent, as shown in Example \ref{ex:order}.

In \cite{DzhenzherSakbaev25-SLLN}, it was proved that the SLLN in WOT holds for random semigroups with i.i.d. uniformly bounded generators in all Banach spaces, while the SLLN in SOT was proved only in uniformly smooth spaces.
It was claimed that by Schur's property, the SLLN in SOT holds for $\ell_1$, although $\ell_1$ is not uniformly smooth.
This argument does not work because the quantifiers ``for all $f$'' and ``almost surely'' cannot be interchanged in general.
However, we show that the conclusion is true for all Banach spaces, thus closing \cite[Conjecture~2.5]{DzhenzherSakbaev25-SLLN}. Moreover, we state the result in a more general way, using norm topology instead of SOT.

For a sequence of Bochner integrable random elements $\set{x_n\colon\Omega\to\mathcal{X}}_{n\in\N}$ in a Banach space $\Xcal$, we define three properties for convenience:
\begin{itemize}
    \item Cesàro boundedness (CB for short), which means $\sup_{n \in \N} \frac{1}{n} \sum_{i=1}^n \norm{x_i} < \infty$ almost surely. This property also makes sense for sequences of deterministic elements regarded as constant random elements.
    \item Cesàro boundedness of expectations (CBE): $\sup_{n \in \N} \frac{1}{n} \sum_{i=1}^n \norm{\Ex x_i} < \infty$.
    \item strong law of large numbers (SLLN): $\frac{1}{n} \sum_{i=1}^n (x_i - \Ex x_i) \to 0$ in the norm of $\Xcal$ almost surely.
\end{itemize}

\begin{theorem}[SLLN in norm topology; proved in~Section~\ref{s:proof} below]\label{t:slln-norm}
    Let \(\Xcal\) be a Banach space.
    Let $\set{A_n\colon\Omega\to\Lcal(\mathcal{X})}_{n\in\N}$ be a sequence of Bochner integrable random operators that has properties CB, CBE and SLLN. Then, the sequence of random semigroups $\set{e^{A_nt}}_{n \in \N}$ satisfies the SLLN in norm topology. If, in addition, $\frac{1}{n}\sum_{i=1}^n \Ex A_i$ converges to $A \in \Lcal(\Xcal)$, then almost surely $e^{A_1t/n} \ldots e^{A_nt/n}$ converges to $e^{At}$ uniformly on any segment $[0, T]$.
\end{theorem}

Theorem \ref{t:kolm-slln} yields the following corollary for i.i.d. random operators:

\begin{corollary}[SLLN for i.i.d. random semigroups]\label{cor:slln-norm}
    Let \(\Xcal\) be a Banach space. For any sequence $\set{A_n\colon\Omega\to\Lcal(\mathcal{X})}_{n\in\N}$ of i.i.d. Bochner integrable random operators, the sequence of random semigroups $\set{e^{A_n t}}_{n \in \N}$ satisfies the SLLN in norm topology. In particular, $e^{A_1t/n} \ldots e^{A_nt/n}$ converges almost surely to $e^{\Ex A_1t}$ uniformly on any segment.
\end{corollary}

\begin{proof}
    We need to check that the hypotheses of Theorem \ref{t:slln-norm} are satisfied, that is, the sequence has the three main properties in this case.
    
    Note that $\set{\norm{A_n}}_{n\in\N}$ are i.i.d. integrable random variables. By the scalar SLLN, almost surely, the sequence $\frac{1}{n}\sum_{i=1}^n \norm{A_i}$ converges to $\Ex \norm{A_1}$; in particular, it is bounded.
    
    The expectations are bounded since the moments are equal.

    The SLLN for random operators follows from Theorem \ref{t:kolm-slln} (as it was said, separability is automatic).
\end{proof}

\section{Proof of Theorem~\ref{t:slln-norm}}\label{s:proof}

In our proof, we consider the Bochner integral over a segment $[a, b]$ (with the Lebesgue measure), which we denote by $\int_a^b x(t) \, dt$. One of the basic properties of the Lebesgue integral generalizes to this case:

\begin{proposition}[Fundamental theorem of calculus for the Bochner integral]\label{l:fundamental}
    Let $\Xcal$ be a Banach space. Let $x \colon [a, b] \to \Xcal$ be continuous on $[a, b]$ and differentiable on $(a, b)$ with $x'$ Bochner integrable on $[a, b]$. Then,
    \[
        \int_a^b x'(t) \,dt = x(b) - x(a).
    \]
\end{proposition}

\begin{proof}
    It is sufficient to show that the equality holds in composition with an arbitrary bounded linear functional. Let $f \in \Xcal^*$. Then, $f \circ x$ is continuous on $[a, b]$ and $(f \circ x)'(t) = f(x'(t))$. Since the Bochner integral commutes with bounded linear operators, we have
    \[
        f\ps{\int_a^b x'(t) \,dt} = \int_a^b f(x'(t))\,dt = \int_a^b (f \circ x)'(t) \,dt.
    \]
    By the Fundamental theorem of calculus for the Lebesgue integral,
    \[
        \int_a^b (f \circ x)'(t) \,dt = f(x(b)) - f(x(a)) = f(x(b) - x(a)).
    \]
    Thus, for all $f \in \Xcal^*$,
    \[
        f\ps{\int_a^b x'(t) \,dt} = f(x(b) - x(a)).
    \]
\end{proof}

\begin{corollary}[Integration by parts]
    Let $\Xcal$ be a Banach space. Let $U, V \colon [a, b] \to \Lcal(\Xcal)$ be continuous on $[a, b]$ and differentiable on $(a, b)$ with $U', V'$ Bochner integrable on $[a, b]$. Then,
    \[
        \int_a^b U(t)V'(t)\,dt = U(b)V(b) - U(a)V(a) - \int_a^b U'(t)V(t)\,dt.
    \]
\end{corollary}

These properties clearly extend to $x$ (or $U$, $V$) being continuous and piecewise differentiable on $[a, b]$.

We will also refer to a measure-theoretic version of a classical integral inequality:
\begin{theorem}[Grönwall's inequality {\cite[Appendix~5, Theorem~5.1]{EthierKurtz1986}}]
    Let $\mu$ be a (non-negative) Borel measure on a finite interval $[a, b]$, and let $f$ be a non-negative bounded Borel measurable function on $[a, b]$ such that
    \[
        f(t) \leq c + \int_{[a, t)} f(s) \, \mu(ds)
    \]
    for some $c \geq 0$ and all $t \in [a, b]$. Then
    \[
        f(t) \leq ce^{\mu([a, t))}
    \]
    for all $t \in [a, b]$.
\end{theorem}

For a general Borel measure, Grönwall's inequality requires that the bounding integral is taken over the interval $[a, t)$ with the endpoint $t$ excluded. This detail is substantial if the measure has atoms; otherwise, integrating over $[a, t]$ instead clearly makes no difference.

The core of the proof of Theorem~\ref{t:slln-norm} lies in the following deterministic lemma:

\begin{lemma}~\label{l:det}
    Let \(\Xcal\) be a Banach space.
    Let $\set{A_n \in \Lcal(\Xcal)}_{n\in\N}$ and $\set{B_n \in \Lcal(\Xcal)}_{n\in\N}$ be two Cesàro bounded sequences such that $\frac{1}{n} \sum_{i=1}^n (A_i - B_i) \to 0$.
    Then for every $T > 0$
    \[
        \lim_{n\to\infty} \sup\limits_{t \in [0,T]} \norm{e^{A_1t/n} \ldots e^{A_{n}t/n} - e^{B_1 t/n} \ldots e^{B_n t/n}} = 0.
    \]
\end{lemma}

\begin{proof}
    Fix $T > 0$. Let $I_k := (\frac{k-1}{n}, \frac{k}{n})$ and $\overline{I_k} := [\frac{k-1}{n}, \frac{k}{n}]$ for $k=1, \dots, n$.
    For $t \in [0, T]$ and $s \in \overline{I_k}$, define
    \[
        P_{A, n}(s, t) = e^{A_1 t/n} \dots e^{A_{k-1} t/n} e^{A_k t\ps{s - \frac{k-1}{n}}}.
    \]
    Note that
    \[
        \frac{\partial P_{A, n}}{\partial s}(s, t) = t P_{A, n}(s, t) A_k
    \]
    for $s \in I_k$. Then, by the fundamental theorem of calculus for the Bochner integral, $P_{A, n}$ satisfies the following integral equation for a variable $s \in [0, 1]$ and a parameter $t \in [0, T]$:
    \begin{gather*}
       P_{A, n}(s, t) = I + \int_0^s t P_{A, n}(\tau, t) G_{A, n}(\tau) \,d\tau,
    \end{gather*}
    where $G_{A, n}(s) = A_k$ for $s \in I_k$. Similarly, define $G_{B, n}$ and $P_{B, n}$, which satisfies the analogous integral equation.

    Put $D_n(s, t) := P_{A, n}(s, t) - P_{B, n}(s, t)$. Note that $P_{A, n}(1, t) = e^{A_1t/n} \ldots e^{A_{n}t/n}$ and $P_{B, n}(1, t) = e^{B_1t/n} \ldots e^{B_{n}t/n}$, so we need to show that $\norm{D_n(1, t)} \to 0$ uniformly on $[0, T]$. Write $D_n(s, t)$ using the integral representation of $P_{A, n}$ and $P_{B, n}$:
    \[
        D_n(s, t) = \int_0^s t \Big( P_{A, n}(\tau, t) G_{A, n}(\tau) - P_{B, n}(\tau, t) G_{B, n}(\tau) \Big) \,d\tau.
    \]
    Add and subtract $t P_{A, n}(\tau, t) G_{B, n}(\tau)$ inside the integral:
    \[
        D_n(s, t) = \int_0^s t D_n(\tau, t) G_{B, n}(\tau) \,d\tau + \underbrace{\int_0^s t P_{A, n}(\tau, t) (G_{A, n}(\tau) - G_{B, n}(\tau)) \,d\tau}_{:= R_n(s, t)}.
    \]
    Taking norms, we get:
    \[
        \norm{D_n(s, t)} \leq \norm{R_n(s, t)} + \int_0^s t \norm{D_n(\tau, t)} \norm{G_{B, n}(\tau)} \,d\tau.
    \]
    Bound $\norm{R_n(s, t)}$ by its supremum over $s \in [0, 1]$ and apply Grönwall's inequality with respect to the variable~$s$ and the absolutely continuous measure $\mu(d\tau) = t\norm{G_{B, n}(\tau)} d\tau$:
    \[
        \|D_n(1, t)\| \leq \left( \sup_{s \in [0,1]} \|R_n(s, t)\| \right) \exp\left(\int_0^1 t\|G_{B, n}(\tau)\| \,d\tau \right).
    \]
    Since $\{A_n\}$ is CB, there is a constant $M$ such that
    \[
        \int_0^1 \| G_{A, n}(\tau)\| \,d\tau = \frac{1}{n} \sum_{i=1}^n \|A_i\| \leq M.
    \]
    Then, for all $s \in [0,1]$ and $t \in [0,T]$ we have
    \[
        \|P_{A, n}(s, t)\| \leq \exp\left(\frac{t}{n} \sum_{i=1}^n \|A_i\|\right) \leq e^{TM} =: C.
    \]
    Similarly, since $\{B_n\}$ is CB, $\int_0^1\|G_{B, n}(\tau)\|\,d\tau \leq M$ (for sufficiently large $M$). Hence,
    \[
        \|D_n(1, t)\| \leq C \sup_{s \in [0,1]} \|R_n(s, t)\|.
    \]
    Define $\Delta_n(s) := \int_0^s (G_{A, n}(\tau) - G_{B, n}(\tau)) \,d\tau$. Rewrite $R_n(s, t)$ using $d\Delta_n(\tau) = (G_{A, n}(\tau) - G_{B, n}(\tau)) \,d\tau$ and integrate by parts:
    \[
        R_n(s, t) = t \int_0^s P_{A, n}(\tau, t) d\Delta_n(\tau) = t P_{A, n}(s, t) \Delta_n(s) - t \int_0^s \frac{\partial P_{A, n}}{\partial \tau}(\tau, t) \Delta_n(\tau) \,d\tau.
    \]
    Substitute $\frac{\partial P_{A, n}}{\partial \tau}(\tau, t) = t P_{A, n}(\tau, t) G_{A, n}(\tau)$:
    \[
        R_n(s, t) = t P_{A, n}(s, t) \Delta_n(s) - t^2 \int_0^s P_{A, n}(\tau, t) G_{A, n}(\tau) \Delta_n(\tau) \,d\tau.
    \]
    Taking norms and substituting the uniform bounds $\|P_{A, n}\| \leq C$ and $t \leq T$,
    \[
        \|R_n(s, t)\| \leq T C \|\Delta_n(s)\| + T^2 C \int_0^s \|G_{A, n}(\tau)\| \|\Delta_n(\tau)\| \,d\tau.
    \]
    We can factor out the supremum of $\Delta_n$:
    \[
        \|R_n(s, t)\| \leq (TC + T^2CM)\sup_{\tau \in [0, 1]} \|\Delta_n(\tau)\|.
    \]
    Notice that for $\tau = k/n$, we have
    \[
        \Delta_n\ps{\frac{k}{n}} = \frac{1}{n}\sum_{i=1}^k(A_i - B_i).
    \]
    Now we use the last assumption that $\norm{\Delta_n(1)} \to 0$. This implies $\max_{k=1,\ldots,n} \norm{\Delta_n(k/n)} \to 0$. Also, $\Delta_n$ is linear on each interval $\overline{I_k}$, so for any $\tau \in \overline{I_k}$ we can bound the norm of $\Delta_n(\tau)$ by the norm at the endpoints:
    \[
        \norm{\Delta_n(\tau)} \leq \max\ps{\norm{\Delta_n\ps{\frac{k-1}{n}}}, \norm{\Delta_n\ps{\frac{k}{n}}}}.
    \]
    Thus, we have the uniform bound:
    \[
        \sup_{\tau \in [0, 1]}\|\Delta_n(\tau)\| = \max_{k=1,\ldots,n}\norm{\Delta_n\ps{\frac{k}{n}}} \to 0.
    \]
    Combining the obtained estimates, we arrive at the desired convergence:
    \[
        \sup_{t \in [0, T]}\|D_n(1, t)\| \leq C \sup_{t \in [0, T]}\sup_{s \in [0,1]} \|R_n(s, t)\| \leq (TC^2 + T^2C^2M)\sup_{\tau \in [0, 1]} \|\Delta_n(\tau)\| \xrightarrow[n \to \infty]{} 0.
    \]
\end{proof}

\begin{proof}[Proof of Theorem~\ref{t:slln-norm}]
    By hypothesis, there are sets $E_1 \in \sigma_\Omega$ where $\sup_{n \in \N} \sum_{i=1}^n \frac{\norm{A_i}}{n} < \infty$ and $E_2 \in \sigma_\Omega$ where $\sum_{i=1}^n \frac{A_i - \Ex A_i}{n} \to 0$ that have probability~1. Then so does $E := E_1 \cap E_2$. By Lemma \ref{l:det}, it holds pointwise on $E$ that for all $T > 0$
    \[
        \lim_{n\to\infty} \sup\limits_{t \in [0,T]} \norm{e^{A_1t/n} \ldots e^{A_{n}t/n} - e^{\Ex A_1 t/n} \ldots e^{\Ex A_n t/n}} = 0.
    \]
    This means, by definition, that the sequence of random semigroups $\set{e^{A_nt}}_{n \in \N}$ satisfies the SLLN in norm topology. For the additional part of the theorem, apply Lemma \ref{l:det} to sequences $\set{\Ex A_n}_{n \in \N}$ and $\set{A}_{n \in \N}$, and use the triangle inequality.
\end{proof}

\section{Counterexamples}\label{s:counterexamples}

The CB and CBE hypotheses in \ref{t:slln-norm} seem quite restrictive because they force both $e^{A_1t/n} \ldots e^{A_{n}t/n}$ and $e^{\Ex A_1 t/n} \ldots e^{\Ex A_n t/n}$ to be bounded (almost surely). However, we show that none of the three conditions can be discarded.

\begin{example}[SLLN hypothesis in theorem \ref{t:slln-norm} is significant]
    Take a sequence of independent random variables $\set{\xi_n}_{n \in \N}$ such that $\Pr\set{\xi_n = n^2} = n^{-2}$ and $\xi_n = 0$ otherwise. Then, $\Ex \xi_n = 1$, but the Borel--Cantelli lemma implies that $\xi_n \neq 0$ only finitely often almost surely, so $\frac{1}{n} \sum_{i=1}^n \abs{\xi_i} \to 0$ a.s. and $\frac{1}{n} \sum_{i=1}^n (\xi_i - \Ex \xi_i) \to -1$ a.s. (Actually, any sequence with properties CB and CBE but not SLLN would suffice.)
    
    Then, in any non-zero Banach space, the sequence of random semigroups generated by $\set{A_n = \xi_n I}_{n \in \N}$ does not satisfy the SLLN even in WOT:
    \[
        e^{A_1t/n} \ldots e^{A_{n}t/n} - e^{\Ex A_1 t/n} \ldots e^{\Ex A_n t/n} = \ps{e^{\frac{t}{n}\sum_{i=1}^n \xi_i} - e^{\frac{t}{n}\sum_{i=1}^n \Ex \xi_i}}I \to (1 - e^t)I.
    \]
\end{example}

The examples for the other two combinations share a common pattern, so let us introduce it beforehand.

\begin{example}\label{e:3}
    Let $\set{\xi_n, \eta_n}_{n \in \N}$ be two sequences of integrable random variables that satisfy the SLLN but
    \[
        \varlimsup_{n \to \infty} \frac{1}{n^2} \abs{\sum_{1 \leq i < j \leq n} \! (\xi_i\eta_j - \Ex\xi_i\Ex\eta_j)} > 0
    \]
    almost surely (such sequences appear in the examples below).
    Consider any Banach space $\Xcal$ with two non-random operators $U, V \in \Lcal(\Xcal)$ such that $U^2 = V^2 = VU = 0$ but $UV \neq 0$. For example, the following operators in $\R^3$ fit:
    \[
        U \sim
        \begin{pmatrix}
            0 & 1 & 0 \\
            0 & 0 & 0 \\
            0 & 0 & 0
        \end{pmatrix}, \quad
        V \sim
        \begin{pmatrix}
            0 & 0 & 0 \\
            0 & 0 & 1 \\
            0 & 0 & 0
        \end{pmatrix}.
    \]
    Consider the following sequence of random operators: $A_{2k-1} := \eta_kV$ and $A_{2k} := \xi_kU$ for $k \in \N$. Due to the properties of $U$ and $V$, for $n = 2k$ we have:
    \begin{align*}
        e^{A_1t/n} \ldots e^{A_nt/n} - e^{\Ex A_1t/n} \ldots e^{\Ex A_nt/n} &= \prod_{i=1}^k \ps{I + \frac{t}{n}\eta_iV} \! \ps{I + \frac{t}{n}\xi_iU} - \prod_{i=1}^k \ps{I + \frac{t}{n}\Ex\eta_iV} \! \ps{I + \frac{t}{n}\Ex\xi_iU} = \\ &= \frac{t}{n}\sum_{i=1}^k (\xi_i - \Ex\xi_i)U + \frac{t}{n}\sum_{i=1}^k (\eta_i - \Ex\eta_i)V + \frac{t^2}{n^2} \! \sum_{1 \leq i < j \leq k} \! (\xi_i\eta_j - \Ex\xi_i\Ex\eta_j)UV.
    \end{align*}
    Because $\set{\xi_n}$ and $\set{\eta_n}$ satisfy the SLLN, the linear terms vanish almost surely. So, assuming $t > 0$,
    \[
        \varlimsup_{n \to \infty} \norm{e^{A_1t/n} \ldots e^{A_nt/n} - e^{\Ex A_1t/n} \ldots e^{\Ex A_nt/n}} \geq \varlimsup_{k \to \infty} \frac{t^2}{4k^2} \abs{\sum_{1 \leq i < j \leq k} \! (\xi_i\eta_j - \Ex\xi_i\Ex\eta_j)}\norm{UV} > 0 \ a.s.
    \]
    Thus, the SLLN fails for $\set{e^{A_nt}}_{n \in \N}$.
\end{example}

\begin{example}[CBE hypothesis in theorem \ref{t:slln-norm} is significant]\label{e:13}
    Let $\set{\xi_n, \eta_n}_{n \in \N}$ be two sequences of jointly independent random variables such that for all $n$, $\Pr\set{\xi_n = (-1)^nn^2} = n^{-3/2}$, $\xi_n = 0$ otherwise, and $\eta_n$ has the same distribution. Note that $\Ex \xi_n = (-1)^n\sqrt{n}$. By the Borel--Cantelli lemma, $\xi_n = \eta_n = 0$ for all but finitely many $n$ almost surely (which implies boundedness). Denote this event by $E$. On $E$, we have
    \[
        \frac{1}{n}\sum_{k=1}^n (\xi_k - \Ex\xi_k) \sim -\frac{1}{n}\sum_{k=1}^n \Ex\xi_k = -\frac{1}{n}\sum_{k=1}^n (-1)^k\sqrt{k} = O\ps{\frac{\sqrt{n}}{n}} \to 0.
    \]
    The same applies to $\eta_n$. Also, observe that on $E$,
    \[
         \frac{1}{n^2} \! \sum_{1 \leq i < j \leq n} \! (\xi_i\eta_j - \Ex\xi_i\Ex\eta_j) \sim -\frac{1}{n^2} \! \sum_{1 \leq i < j \leq n} \! \Ex\xi_i\Ex\xi_j = -\frac{1}{2}\ps{\frac{1}{n}\sum_{k=1}^n \Ex\xi_k}^2 + \frac{1}{2n^2}\sum_{k=1}^n (\Ex\xi_k)^2 \sim \frac{1}{2n^2}\sum_{k=1}^n k \sim \frac{1}{4}.
    \]
    
    Now, consider the random operators $\set{A_n}_{n \in \N}$ from Example \ref{e:3}. This sequence is bounded almost surely and satisfies the SLLN because $\{\xi_n\}$ and $\{\eta_n\}$ do. Additionally, the operators are independent. However, as the example shows, the sequence $\set{e^{A_nt}}_{n \in \N}$ does not satisfy the SLLN.
\end{example}

\begin{example}[CB hypothesis in theorem \ref{t:slln-norm} is significant]
    Let $\set{\xi_n, \eta_n}_{n \in \N}$ be two sequences of jointly independent random variables such that for all $n$, both $\xi_n$ and $\eta_n$ are distributed as follows:
    \[
        \begin{cases}
            \Pr\set{\xi_n = (-1)^n\sqrt{n}(1 - n^2)} = n^{-2} \\
            \Pr\set{\xi_n = (-1)^n\sqrt{n}} = 1 - n^{-2}
        \end{cases}.
    \]
    Here, $\Ex \xi_n = \Ex \eta_n = 0$, but the Borel--Cantelli lemma ensures that $\xi_n = \eta_n = (-1)^n\sqrt{n}$ for all but finitely many $n$ on a set $E \in \sigma_{\Omega}$ of probability 1 (which is opposite to the previous example). Then, on $E$, we similarly have $\frac{1}{n}\sum_{i=1}^n \xi_i \to 0$, $\frac{1}{n}\sum_{i=1}^n \eta_i \to 0$, and
    \[
         \frac{1}{n^2} \! \sum_{1 \leq i < j \leq n} \! \xi_i\eta_j \sim \frac{1}{n^2} \! \sum_{1 \leq i < j \leq n} \! (-1)^{i+j}\sqrt{ij} = \frac{1}{2}\ps{\frac{1}{n}\sum_{k=1}^n (-1)^k\sqrt{k}}^2 - \frac{1}{2n^2}\sum_{k=1}^n k \sim - \frac{1}{2n^2}\sum_{k=1}^n k \sim -\frac{1}{4}.
    \]
    
    Then, the sequence $\set{A_n}_{n \in \N}$ from Example \ref{e:3} satisfies the SLLN and $\Ex A_n = 0$ for all $n$, but the corresponding sequence of random semigroups does not satisfy the SLLN.
\end{example}

We also show that the direct analog of Theorem \ref{t:slln-norm} for SOT or WOT does not hold in general, even for independent random operators.

\begin{example}\label{ex:sot}
    We construct a sequence of random operators $\set{A_n}_{n \in \N}$ in $\Xcal = \ell_2(\N_0)$ (which can be replaced with any infinite-dimensional Banach space that admits a Schauder basis) with the following properties:
    \begin{itemize}
        \item the random operators are independent;
        \item CB;
        \item CBE;
        \item for all $x \in \Xcal$, the sequence $\set{A_nx}_{n \in \N}$ satisfies the SLLN;
        \item the sequence of random semigroups $\set{e^{A_nt}}_{n \in \N}$ does not satisfy the SLLN in WOT.
    \end{itemize}

    For \(m \in \N\), define \(K_m, R_m \in \Lcal(\Xcal)\) as
    \[
        K_m x := x_m e_0, \quad R_m x := x_0 e_m - x_m e_0.
    \]
    Choose a sequence $\set{n_m}_{m \in \N}$ such that
    \[
        \frac{\sum_{k=1}^{m-1}n_k}{n_m} \to 0, \quad n_m > 2n_{m-1},
    \]
    e.g., $n_m = 2^{m^2}$.
    Let \(\set{\varepsilon_m}_{m \in \N}\) be a sequence of independent Rademacher variables, i.e., $\Pr\set{\varepsilon_m = 1} = \Pr\set{\varepsilon_m = -1} = 1/2$. Define
    \[
        A_{n_m} := n_m \varepsilon_m K_m, \quad A_{2n_m} := n_m R_m,
    \]
    and $A_n := 0$ for other $n$.
    The random operators \(\set{A_n}_{n \in \N}\) are independent since the only non-constant ones depend on distinct $\varepsilon_m$.
    
    We have $\Ex A_{2n_m} = n_mR_m$ and $\Ex A_n = 0$ for other $n$, so the growth rate of $n_m$ gives
    \[
        \sup_{n \in \N} \frac{1}{n} \sum_{i=1}^n \norm{\Ex A_i} = \sup_{m \in \N} \frac{1}{2n_m} \sum_{k=1}^{m} \norm{\Ex A_{2n_k}} = \sup_{m \in \N} \frac{1}{2n_m} \sum_{k=1}^{m} n_k < \infty.
    \]
    Similarly,
    \[
        \sup_{n \in \N} \frac{1}{n} \sum_{i=1}^n \norm{A_i} \leq \sup_{m \in \N} \frac{1}{n_m} \sum_{k=1}^{m} (\norm{A_{n_k}} + \norm{A_{2n_k}}) = \sup_{m \in \N} \frac{2}{n_m} \sum_{k=1}^{m} n_k < \infty.
    \]
    Also, for every $x$ and $n \in [n_m, n_{m+1})$,
    \[
        \frac{1}{n} \norm{\sum_{i=1}^n (A_i-\Ex A_i)x} = \frac{1}{n} \norm{\sum_{k=1}^m n_k \varepsilon_k x_k e_0} \leq \frac{1}{n_m} \ps{\sum_{k=1}^{m-1} n_k\norm{x} + n_m \abs{x_m}} = \frac{1}{n_m} \sum_{k=1}^{m-1} n_k\norm{x} + \abs{x_m} \xrightarrow[n \to \infty]{} 0.
    \]
    So, the weakened assumptions are satisfied.
    
    For $m \geq 2$, denote
    \begin{alignat*}{2}
        &P_n(t) := e^{A_1t/n} \ldots e^{A_{n}t/n}, && \qquad U_m(t) := P_{2n_{m-1}}\!\ps{\frac{n_{m-1}}{n_m}t}, \\
        &Q_n(t) := e^{\Ex A_1 t/n} \ldots e^{\Ex A_n t/n}, && \qquad V_m(t) := Q_{2n_{m-1}}\!\ps{\frac{n_{m-1}}{n_m}t}.
    \end{alignat*}
    I.e., $U_m$ and $V_m$ are the products of the first $2n_{m-1}$ factors in $P_{2n_m}$ and $Q_{2n_m}$ respectively. Then, for $m \geq 2$, the following decomposition holds:
    \[
        P_{2n_m}(t) - Q_{2n_m}(t) = U_m(t) e^{\varepsilon_m K_m t/2} e^{R_m t/2} - V_m(t) e^{R_m t/2}.
    \]
    Since $K_m^2 = 0$, the factor $e^{\varepsilon_m K_m t/2}$ reduces to $I + \varepsilon_m K_m t/2$, so we can group the terms as
    \[
        P_{2n_m}(t) - Q_{2n_m}(t) = \ps{U_m(t) - V_m(t)} e^{R_m t/2} + \frac{t}{2} \varepsilon_m U_m(t) K_m e^{R_m t/2}.
    \]
    From the mean value inequality in Banach spaces, we get
    \[
        \norm{U_m(t) - I} \leq \frac{t}{2n_m} \ps{\sum_{i=1}^{2n_{m-1}} \norm{A_i}} \prod_{i=1}^{2n_{m-1}} e^{\norm{A_i} t/(2n_m)} = \frac{t}{n_m} \ps{\sum_{k=1}^{m-1} n_k} e^{(t/n_m) \sum_{k=1}^{m-1} n_k} \xrightarrow[m \to \infty]{} 0,
    \]
    and
    \[
        \norm{V_m(t) - I} \leq \frac{t}{2n_m} \ps{\sum_{i=1}^{2n_{m-1}} \norm{\Ex A_i}} \prod_{i=1}^{2n_{m-1}} e^{\norm{\Ex A_i} t/(2n_m)} = \frac{t}{2n_m} \ps{\sum_{k=1}^{m-1} n_k} e^{t/(2n_m) \sum_{k=1}^{m-1} n_k} \xrightarrow[m \to \infty]{} 0.
    \]
    Now, it is easy to see that the SLLN in WOT fails on the vector $e_0$ and the coordinate functional $f_0$. Indeed, note that
    \[
        e^{R_m t/2} e_0 = \cos(t/2)e_0 + \sin(t/2)e_m.
    \]
    So, using the triangle inequality, the estimates above and the fact that $\|K_m\| = \|R_m\| = 1$, we can write
    \[
        f_0\ps{\ps{P_{2n_m}(t) - Q_{2n_m}(t)}e_0} \sim \frac{t}{2} \varepsilon_m f_0\ps{K_m e^{R_m t/2} e_0} = \frac{t}{2} \varepsilon_m \sin(t/2).
    \]
    So, for $t \notin 2\pi\Z$, the convergence in WOT fails. 
\end{example}

Lastly, we show that the order of composition in the definition of the SLLN for random semigroups is significant in general.

\begin{example}\label{ex:order}
     Let $\Xcal = \R^2$. Take linear operators $A \colon (x, y) \mapsto (y, 0)$ and $B \colon (x, y) \mapsto (x, 0)$ on $\Xcal$. Set $A_1 = \varepsilon A$ for $\varepsilon$ a Rademacher random variable and $A_n = (n-1)B$ for $n \geq 2$. Then,
\[
    e^{A_1t/n} \ldots e^{A_{n}t/n} - e^{\Ex A_1 t/n} \ldots e^{\Ex A_n t/n} = \ps{e^{\varepsilon At/n} - e^{\Ex\varepsilon At/n}} e^{\sum_{k=2}^n B(k-1)t/n} = \frac{t}{n}\varepsilon Ae^{Bt(n-1)/2} =: D_n(t).
\]
We have $D_n(t)(x, y) = \frac{t}{n}\varepsilon A(e^{t(n-1)/2}x, y) = \frac{t}{n}\varepsilon(y, 0)$, so $\sup_{t \in [0, T]} \norm{D_n(t)} \to 0$. However, the reversed order of composition gives
\[
    e^{A_nt/n} \ldots e^{A_{1}t/n} - e^{\Ex A_n t/n} \ldots e^{\Ex A_1 t/n} = e^{\sum_{k=2}^n B(k-1)t/n} \ps{e^{\varepsilon At/n} - e^{\Ex\varepsilon At/n}} = \frac{t}{n}\varepsilon e^{Bt(n-1)/2}A =: E_n(t).
\]
Here, $E_n(t)(x, y) = \frac{t}{n}\varepsilon e^{Bt(n-1)/2}(y, 0) = \frac{t}{n}\varepsilon(e^{t(n-1)/2}y, 0)$, so $\sup_{t \in [0, T]} \norm{E_n(t)} \to \infty$ for all $T > 0$.

For a counterexample to the other implication, define $A \colon (x, y) \mapsto (0, x)$ instead.
\end{example}

\section{Remarks and further research}\label{s:conclusions}

The proof of Theorem \ref{t:slln-norm} actually gives the following estimate for some $C = C(T, \omega)$:
\[
    \sup\limits_{t \in [0,T]} \norm{e^{A_1t/n} \ldots e^{A_{n}t/n} - e^{\Ex A_1 t/n} \ldots e^{\Ex A_n t/n}} \leq \frac{C}{n} \max_{k=1,\ldots,n} \norm{\sum_{i=1}^k (A_i - \Ex A_i)}.
\]
Thus, the left-hand side automatically inherits any asymptotic upper bound for the right-hand side. For the latter, there is a classical result due to de Acosta \cite{deAcosta1981} and independently to Azlarov and Volodin \cite{AzlarovVolodin1981}, which was initially proved for random variables by Marcinkiewicz and Zygmund \cite{MarcinkiewiczZygmund1937}:

\begin{theorem}
    Let $\Xcal$ be a separable Banach space and let $p \in [1, 2)$. The following conditions are equivalent:
    \begin{enumerate}
        \item $\Xcal$ is of Rademacher type $p$ (see \cite{Ledoux-Talagrand-1991-probability} for the definition);
        \item for any sequence of i.i.d. $\Xcal$-valued random elements $\set{x_n}_{n \in \N}$ with $\Ex x_n = 0$ and $\Ex\norm{x_n}^p < \infty$, $\frac{1}{n^{1/p}}\sum_{i=1}^n x_i \to 0$ almost surely.
    \end{enumerate}
\end{theorem}

So, if $\set{A_n}_{n \in \N}$ is a sequence of i.i.d. Bochner integrable random operators such that $\Ex\norm{A_n}^p < \infty$, and the closed linear span of the essential range of $A_1$ is of Rademacher type $p$, then almost surely $\frac{1}{n^{1/p}}\sum_{i=1}^n (A_i - \Ex A_i) \to 0$. By the standard argument, this implies $\frac{1}{n^{1/p}} \max_{k=1,\ldots,n} \norm{\sum_{i=1}^k (A_i - \Ex A_i)} \to 0$, which gives the estimate
\[
    \sup\limits_{t \in [0,T]} \norm{e^{A_1t/n} \ldots e^{A_{n}t/n} - e^{\Ex A_1t}} = o(n^{\frac{1}{p} - 1}).
\]
However, one should keep in mind that Rademacher type $p > 1$ is not automatic here unless the space is finite-dimensional. The essential range is necessarily separable due to Bochner measurability of $A_n$, but this does not imply non-trivial Rademacher type.

In practice, the assumptions of Theorem \ref{t:slln-norm} might be infeasible to verify directly. To obtain special cases of the theorem, one imposes additional restrictions on either the range of the random operators or their measure-theoretic properties. The former includes independent uniformly bounded operators whose essential ranges are contained in a common $B$-convex subspace or a compact subset of $\Lcal(\Xcal)$. The latter includes Corollary \ref{cor:slln-norm} and stationary ergodic sequences of random operators (see \cite[Section~3]{steele1989} for the corresponding ergodic theorem in Banach spaces). In general, there is a slightly stronger set of sufficient conditions:
\begin{gather*}
    \sup_{n \in \N} \frac{\sum_{i=1}^n \Ex \norm{A_i}}{n} < \infty, \\
    \frac{\sum_{i=1}^n (\norm{A_i} - \Ex \norm{A_i})}{n} \xrightarrow[]{a.s.} 0, \\
    \frac{\sum_{i=1}^n (A_i - \Ex A_i)}{n} \xrightarrow[]{a.s.} 0,
\end{gather*}
where the second condition could be derived from one of the SLLN-type theorems for random variables. However, as already mentioned, the SLLN in Banach spaces may fail even if the random elements are independent and uniformly bounded but not identically distributed.

For non-identically distributed random semigroups, moving from norm topology to SOT or WOT could prove beneficial. Example \ref{ex:sot} looks somewhat unnatural because it exploits very rare spikes of linear amplitude in the sequence. We believe that the SLLN in SOT and WOT can still be generalized under some stronger asymptotic assumptions. Moreover, it is reasonable to expect some connection between the SLLN for non-identically distributed random semigroups and the properties of the Banach space, as it exists in the case of the SLLN for random elements.

\section*{Acknowledgements}

The author is grateful to S.\,V. Dzhenzher for valuable discussion and review throughout the creation of this paper.